\documentclass[12pt]{amsart}
\usepackage[T1]{fontenc}
\usepackage{amssymb}
\usepackage{geometry}
\usepackage{tikz}
\usepackage{float}
\usepackage{url}
\usepackage{mathtools}
\usepackage{amsmath,amsfonts,amsthm}
\usepackage{mathrsfs} 
\newcommand{\irr}{\operatorname{irr}}
\newcommand{\case}[1]{\paragraph*{Case #1:}}
\newtheorem{hypothesize}{Hypothesize}
\usepackage{multirow}
\usepackage{subfig}

\usepackage{amsmath}

\newtheorem{theorem}{Theorem}[section]
\newtheorem{lemma}[theorem]{Lemma}
\newtheorem{proposition}{Proposition}[section]
\newtheorem{corollary}[theorem]{Corollary}

\newtheorem{example}{Example}

\author{Jasem Hamoud}
\address{\textbf{Jasem Hamoud} 
Department of Discrete Mathematics, Moscow Institute of Physics and Technology}
\email{khamud@phystech.edu}
\thanks{Proceedings of the 67th All-Russian Scientific Conference of MIPT, March 31 – April 5, 2025 Applied mathematics and computer science.  Moscow: Fizmatkniga, 2025. — 312 p.) ISBN 978-5-89155-446-7. Available: \url{https://old.mipt.ru/upload/medialibrary/444/5_-fpmi_2025.pdf}, p.195}

\author{Artem Kornosov}
\address{\textbf{Artem Kornosov} Department of Discrete Mathematics, Moscow Institute of Physics and Technology}
\email{kurnosov@phystech.edu}
\thanks{}
\title[Degree Sequence of Albertson]{Degree Sequence of Albertson and $\sigma$-Indices  on Trees of Order $n\geqslant 3$ }

\date{}
\begin{document}

\begin{abstract}
  In this paper, we presented a study of topological indices on trees, where we show a relationship with irregularity of Albertson index and minimum, maximum degrees $\delta,\Delta$ of graph $G$, where contribute vital roles in determining connection, shading, component incorporation, and realisability where well-known Albertson index as: $\operatorname{irr}(G)=\sum_{uv\in E(G)}\lvert d_u(G)-d_v(G) \rvert$. The sigma index on trees that we introduced as $\sigma(T)=(d_1-1)^3+\sum_{i=1}^{3}(d_i-1)(d_i-2)+(d_3-1)^3$.
\end{abstract}

\maketitle

\noindent\rule{12.7cm}{1.0pt}


\medskip

\noindent

\medskip


\noindent\rule{12.7cm}{1.0pt}

\section{Introduction}
Throughout this paper, let $G=(V,E)$ be a simple, connected graph, denote by $uv$ an edge connecting vertices $u$ and $v$. Further, let $\deg i=\deg(i)$ the degree of a vertex $i$. In 1969, Harary, F. in~\cite{Harary} mention to a diverge at a vertex $u$ in a tree $T$ is its largest subtree which includes $u$ as a termination. Therefore, the amount of diverge at $u$ equals $\deg u$, which for any tree has a center chord of a maximum of one or two points. In 1971,  let be refer to  the maximal number of edges by $\mathcal{E}$, Wai-Kai Chen in~\cite{Wai} refer to the graph $G$ has $n$ nodes and $c$ components, then $\mathcal{E}=\frac{1}{2}(n-c)(n-c+1)$, we can make it as a general case if  $G_{i}, i=1,2, \ldots, c$, is the components of $G$, each having $n_{i}$ nodes, then $\mathcal{E}=\frac{1}{2} n_{i}\left(n_{i}-1\right)$. Further, in 2009, Voloshin, V. I, in~\cite{Voloshin} consider a graph $G=(X, E)$, determine number known well as Szekeres-Wilf number:
\[
M(G)=\max _{X^{\prime} \subseteq X} \min _{x \in G^{\prime}} d(x) .
\]
In 1997,  M. O. Albertson in~\cite{ALBERTSON} mention to the imbalance of an edge $uv$ by $imb(uv)=\lvert d_u-d_v\rvert$, Considered graph $G$ regular if all of its vertices have the exact same degree, then irregularity measure define in~\cite{ALBERTSON, GUTMAN, Brandt} as: 
\[
\operatorname{irr}(G)=\sum_{uv\in E(G)}\lvert d_u(G)-d_v(G) \rvert.
\]
Dorjsembe, S et al. In~\cite{Dorjsembe} given a relationship with irregularity of Albertson index and minimum, maximum degrees $\delta,\Delta$ of graph $G$, where contribute vital roles in determining connection, shading, component incorporation, and realisability. Among this degrees, which reflect the maximum and minimum number of edges incident to every vertex in the graph, are important elements of a graph's sophistication and performance. It is given by: $\operatorname{irr} (G)> \frac{\delta(\Delta-\delta)^2.|V|}{\Delta+1}$, the degree sum formula (or handshaking lemma) is know well as doubles of edges as $2|E|$,  the first and the second Zagreb index, $M_1(G)$ and $M_2(G)$ are defined in~\cite{TRINAJSTIC, WILCOX,ALBERTSON} as: 
\[
M_1(G)=\sum_{i=1}^{n}d_i^2, \quad \text{and} \quad M_2(G)=\sum_{uv\in E(G)} d_u(G)d_v(G).
\]
In~\cite{Hosam,Brandt} define total irregularity as $\operatorname{irr}_T(G)=\sum_{\{u,v\} \subseteq V(G)}\lvert d_u(G)-d_v(G) \rvert$. When  $G$ is a tree, Ghalavand, A. et al. In~\cite{Ghalavand} proved  $\operatorname{irr}_T(G) \leq \frac{n^2}{4} \operatorname{irr}(G)$,  in this case, they show $\operatorname{irr}_T(G) \leq(n-2) \operatorname{irr}(G)$. The recently introduced $\sigma(G)$ irregularity index is a simple diversification of the previously established Albertson irregularity index, in~\cite{GutmanTogan,DimitrovAbdo} defined as: 
\[
\sigma(G)=\sum_{uv\in E(G)}\left( d_u(G)-d_v(G) \right)^2.
\]
This paper is organized as follows. In Section~\ref{sec2}, we observe the important concepts to our work including literature view of most related papers, in Section~\ref{sec3} we computing the main result of this paper including Albertson and Sigma indices on caterpillar tree, 
\section{Preliminaries}\label{sec2}
In this section, we show in Lemma~\ref{lem3} the Albertson index with maximum and minimum degree $\Delta,\delta$, also we explain the relationship between Albertson index and Sigma index on star graph $S_n$ as we show in Lemma~\ref{lem1}, \ref{lem2}. In Proposition~\ref{proe1}, Collatz and Bell in~\cite{Collatz, Bell} show the maximum of Albertson index on the set of all $n$ vertex graph. 
\begin{lemma}\cite{DimitrovAbdo,Marjan}~\label{lem3}
Let $G(V,E)$ a graph (simple  and connected), then: 
	\[
	\operatorname{irr}(G) \leq \frac{(\Delta-\delta)}{\sqrt{\Delta\delta}}\sqrt{|E|\sum_{uv\in E(G)}^{} d_G(u)d_G(v)}. 
	\]
\end{lemma}
The sigma index, a graph-theoretic estimation of irregularity, is now recognised as an important tool for investigating architectural variance in networks. This index, estimated as the sum of squared degree differences between adjacent vertices.
\begin{theorem}\cite{Sekar}
Let be a complete bipartite graph is \( K_{m, n} \), then sigma index of  graph is \( \sigma(K_{m, n})=(m n)(n-m)^{2} \).
\end{theorem}
\begin{theorem}\cite{GutmanTogan}
For any double star graph $S_{r, k}$, the degrees of two adjacent central vertices $u,v$ defined as   $d_{u}=k \geq 3, d_{v}=r \geq 1$. Then the sigma index of \( S_{r, k} \) is given by
\[
\sigma\left(S_{r, k}\right)=(k-1)^{3}+k^{2}+(r-1)^{3}+r^{2}-2 k r .
\]
\end{theorem}
\begin{proposition}~\label{proe1}
 Let be $\mathscr{G}(n)$ is the  set of  all $n$ vertex on graph $G$,  $\lambda(G)$ the index  of the largest eigenvalue in graph, and let be $\bar{d}(G)$ the mean of the vertex degrees, then
    \[
    \operatorname{irr}(G)=\lambda(G)-\bar{d}(G).
    \]
Then, we have: 
\[
\max \{\operatorname{irr}(G): G \in \mathscr{G}(n)\}=\left\{\begin{array}{ll}\frac{1}{4} n-\frac{1}{2} & (n \text { even }) \\ \frac{1}{4} n-\frac{1}{2}+\frac{1}{4 n} & (n \text { odd }).\end{array}\right. 
\].
\end{proposition}
This result implies in~\cite{Bell,Collatz,ALBERTSON, GUTMAN}. This relationship can be demonstrated for the indices in both Lemma~\ref{lem1}, \ref{lem2}, which depict the star graph.
\begin{lemma}\cite{Nasiri}~\label{lem1}
	The star $S_n$  is the only tree of order n that has a great deal of irregularity, satisfying:
	\[\operatorname{irr}\left( {{S_n}} \right) = \left( {n - 2} \right)\left( {n - 1} \right).\]
\end{lemma} 
\begin{lemma}~\label{lem2}
Let be $\mathcal{T}$ a classes of trees, then $\sigma_{max}, \sigma_{min}$ in trees with $n$ vertices by:
 \[\sigma(\mathcal{T})  = \left\{ \begin{array}{l}
 	{\sigma _{\max }(\mathcal{T})} = \left( {n - 1} \right)\left( {n - 2} \right){\rm{        }};n \ge 3\\
 	{\sigma _{\min }(\mathcal{T})} = 0{\rm{                             ; n = 2}.}
 \end{array} \right. \]
 \end{lemma}
 According to a tree with $n$ vertices has $n-1$ edges, then we have $\sigma_{\max}(T)=n-2$. Thus, the modified total Sigma as $\sigma_t(T)=\frac{1}{2}\sum_{(u,v)\subseteq V(T)}\left(d_u-d_v\right)^2$, show that in~\cite{DarkoDimitrov,m1} and in Theorem~\ref{paa}, so that $\sigma_{t,\max}=(n-1)(n-2)^2$.

\begin{theorem}\cite{m1}\label{paa} Let be $G$ is a connected graph of order $n$ which $n\in \mathbb{N}, n\geq 3$ vertices it holds
	$$
	\sigma_{t}(G)\leq \left\{
	\begin{array}{lc} \lceil \frac{n}{4}\rceil \cdot \lfloor \frac{3n}{4}\rfloor \left(n-1-\lceil \frac{n}{4}\rceil \right)^{2}  , \mbox { if\;} \; n \equiv 0 \bmod 4 \mbox { or\; } n \equiv 3 \bmod 4; 
		\\
		\lfloor \frac{n}{4}\rfloor \cdot \lceil \frac{3n}{4}\rceil \left(n-1-\lfloor \frac{n}{4}\rfloor \right)^{2} , \mbox { if\;}\; n \equiv 1 \bmod 4 \mbox { or\; } n \equiv 2 \bmod 4. \\
	\end{array}
	\right.
	$$
\end{theorem}
\section{Main Result}~\label{sec3}
 We mention by $C(n,m)$ caterpillar tree with $(n,m)$ vertices, and refer to $n$ main vertices in figure~\ref{figcater} as we show. 
 \begin{figure}[H]
\centering
\begin{tikzpicture}[scale=.9]
	\draw  (2,2)-- (4,2);
	\draw  (4,2)-- (6,2);
	\draw [line width=2pt,dash pattern=on 1pt off 1pt] (6,2)-- (8,2);
	\draw  (2,2)-- (1.35,1);
	\draw  (2,2)-- (2,1);
	\draw  (2,2)-- (2.77,1.02);
	\draw  (4,2)-- (3.59,1);
	\draw  (4,2)-- (4,1);
	\draw  (4,2)-- (4.55,0.98);
	\draw  (6,2)-- (5.59,1.06);
	\draw  (6,2)-- (6,1);
	\draw  (6,2)-- (6.53,1.04);
	\draw (8,2)-- (7.55,1.02);
	\draw  (8,2)-- (8,1);
	\draw  (8,2)-- (8.59,1.02);
	\draw  (3,2)-- (2,3);
	\draw  (3,2)-- (3,3);
	\draw  (3,2)-- (4,3);
	\draw (1.5772463770998264,2.3647403536785268) node[anchor=north west] {$x_1$};
	\draw (2.8811628144708608,2.006163333401493) node[anchor=north west] {$x_2$};
	\draw (3.8699661128105616,2.538595878661331) node[anchor=north west] {$x_3$};
	\draw (4.869635381461688,2.0170293037129183) node[anchor=north west] {$x_4$};
	\draw (7.879509157726492,2.484266027104205) node[anchor=north west] {$x_n$};
	\draw  (5,2)-- (4.325669695313298,3.0458367583511627);
	\draw  (5,2)-- (5,3);
	\draw (5,2)-- (5.66201453448175,3.029927891218205);
	\draw (5.836706739178537,2.5711937895956067) node[anchor=north west] {$x_5$};
	\begin{scriptsize}
		\draw [fill=black] (2,2) circle (1pt);
		\draw [fill=black] (4,2) circle (1pt);
		\draw [fill=black] (6,2) circle (1pt);
		\draw [fill=black] (8,2) circle (1pt);
		\draw [fill=black] (1.35,1) circle (1pt);
		\draw [fill=black] (2,1) circle (1pt);
		\draw [fill=black] (2.77,1.02) circle (1pt);
		\draw [fill=black] (3.59,1) circle (1pt);
		\draw [fill=black] (4,1) circle (1pt);
		\draw [fill=black] (4.55,0.98) circle (1pt);
		\draw [fill=black] (5.59,1.06) circle (1pt);
		\draw [fill=black] (6,1) circle (1pt);
		\draw [fill=black] (6.53,1.04) circle (1pt);
		\draw [fill=black] (7.55,1.02) circle (1pt);
		\draw [fill=black] (8,1) circle (1pt);
		\draw [fill=black] (8.59,1.02) circle (1pt);
		\draw [fill=black] (3,2) circle (1pt);
		\draw [fill=black] (2,3) circle (1pt);
		\draw [fill=black] (3,3) circle (1pt);
		\draw [fill=black] (4,3) circle (1pt);
		\draw [fill=black] (5,2) circle (1pt);
		\draw [fill=black] (4.325669695313298,3.0458367583511627) circle (1pt);
		\draw [fill=black] (5,3) circle (1pt);
		\draw [fill=black] (5.66201453448175,3.029927891218205) circle (1pt);
	\end{scriptsize}
\end{tikzpicture}
\caption{Example of Caterpillars in graph theory.}\label{figcater}
 \end{figure}
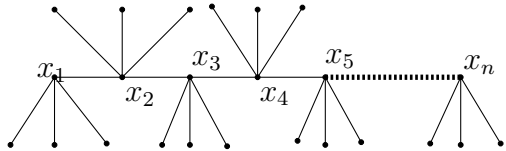
We have $X=(x_1,x_2,\dots,x_n)$ where $n=|X|$, and refer to $m$ a set of leaves adjacent to main vertices. Zhang, P., Wang, X in\cite{j1} show ``Caterpillar trees'' are employed in chemical graph theory for describing molecular structures, and several topological indices, which include the Albertson index, have been used to study the characteristics of molecules. 

  In Figure~\ref{fig2} we show some caterpillars trees $C(n,3)$ for $n=3,4,5,6$: 

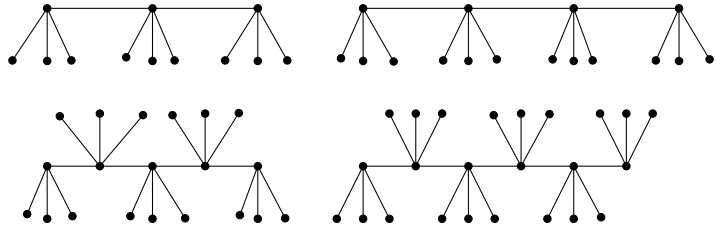
\begin{figure}[H]
\centering
	\begin{tikzpicture}[scale=.7]
		\draw  (1,5)-- (3,5);
		\draw  (3,5)-- (5,5);
		\draw  (1,5)-- (0.34,4.01);
		\draw  (1,5)-- (1,4);
		\draw  (1,5)-- (1.46,4.01);
		\draw  (3,5)-- (2.5,4.07);
		\draw  (3,5)-- (3,4);
		\draw  (3,5)-- (3.42,4.01);
		\draw  (5,5)-- (4.38,4.01);
		\draw  (5,5)-- (5,4);
		\draw  (5,5)-- (5.56,4.01);
		\draw  (7,5)-- (9,5);
		\draw  (9,5)-- (11,5);
		\draw  (11,5)-- (13,5);
		\draw  (7,5)-- (6.58,4.05);
		\draw  (7,5)-- (7,4);
		\draw  (7,5)-- (7.58,3.99);
		\draw  (9,5)-- (8.52,4.01);
		\draw  (9,5)-- (9,4);
		\draw  (9,5)-- (9.54,4.03);
		\draw  (11,5)-- (10.6,4.03);
		\draw  (11,5)-- (11,4);
		\draw  (11,5)-- (11.36,4.01);
		\draw  (13,5)-- (12.56,4.01);
		\draw  (13,5)-- (13,4);
		\draw  (13,5)-- (13.58,4.03);
		\draw  (1,2)-- (2,2);
		\draw  (2,2)-- (3,2);
		\draw  (3,2)-- (4,2);
		\draw  (4,2)-- (5,2);
		\draw  (1,2)-- (0.62,1.09);
		\draw  (1,2)-- (1,1);
		\draw  (1,2)-- (1.48,1.05);
		\draw  (2,2)-- (1.24,2.95);
		\draw  (2,2)-- (2,3);
		\draw  (2,2)-- (2.82,2.97);
		\draw  (3,2)-- (2.58,1.05);
		\draw  (3,2)-- (3,1);
		\draw  (3,2)-- (3.62,1.01);
		\draw  (4,2)-- (3.38,2.97);
		\draw  (4,2)-- (4,3);
		\draw (4,2)-- (4.64,3.01);
		\draw  (5,2)-- (4.66,1.07);
		\draw  (5,2)-- (5,1);
		\draw  (5,2)-- (5.52,1.01);
		\draw  (7,2)-- (8,2);
		\draw  (8,2)-- (9,2);
		\draw  (9,2)-- (10,2);
		\draw  (10,2)-- (11,2);
		\draw  (11,2)-- (12,2);
		\draw  (7,2)-- (6.5,1);
		\draw  (7,2)-- (7,1);
		\draw  (7,2)-- (7.5,1);
		\draw  (8,2)-- (7.5,3);
		\draw  (8,2)-- (8,3);
		\draw  (8,2)-- (8.5,3);
		\draw  (9,2)-- (8.5,1);
		\draw  (9,2)-- (9,1);
		\draw  (9,2)-- (9.5,1);
		\draw (10,2)-- (9.48,2.97);
		\draw  (10,2)-- (10,3);
		\draw  (10,2)-- (10.54,2.99);
		\draw  (11,2)-- (10.5,1);
		\draw  (11,2)-- (11,1);
		\draw  (11,2)-- (11.52,1.03);
		\draw  (12,2)-- (11.5,3);
		\draw (12,2)-- (12,3);
		\draw  (12,2)-- (12.5,3);
		\begin{scriptsize}
			\draw [fill=black] (1,5) circle (2pt);
			\draw [fill=black] (3,5) circle (2pt);
			\draw [fill=black] (5,5) circle (2pt);
			\draw [fill=black] (0.34,4.01) circle (2pt);
			\draw [fill=black] (1,4) circle (2pt);
			\draw [fill=black] (1.46,4.01) circle (2pt);
			\draw [fill=black] (2.5,4.07) circle (2pt);
			\draw [fill=black] (3,4) circle (2pt);
			\draw [fill=black] (3.42,4.01) circle (2pt);
			\draw [fill=black] (4.38,4.01) circle (2pt);
			\draw [fill=black] (5,4) circle (2pt);
			\draw [fill=black] (5.56,4.01) circle (2pt);
			\draw [fill=black] (7,5) circle (2pt);
			\draw [fill=black] (9,5) circle (2pt);
			\draw [fill=black] (11,5) circle (2pt);
			\draw [fill=black] (13,5) circle (2pt);
			\draw [fill=black] (6.58,4.05) circle (2pt);
			\draw [fill=black] (7,4) circle (2pt);
			\draw [fill=black] (7.58,3.99) circle (2pt);
			\draw [fill=black] (8.52,4.01) circle (2pt);
			\draw [fill=black] (9,4) circle (2pt);
			\draw [fill=black] (9.54,4.03) circle (2pt);
			\draw [fill=black] (10.6,4.03) circle (2pt);
			\draw [fill=black] (11,4) circle (2pt);
			\draw [fill=black] (11.36,4.01) circle (2pt);
			\draw [fill=black] (12.56,4.01) circle (2pt);
			\draw [fill=black] (13,4) circle (2pt);
			\draw [fill=black] (13.58,4.03) circle (2pt);
			\draw [fill=black] (1,2) circle (2pt);
			\draw [fill=black] (2,2) circle (2pt);
			\draw [fill=black] (3,2) circle (2pt);
			\draw [fill=black] (4,2) circle (2pt);
			\draw [fill=black] (5,2) circle (2pt);
			\draw [fill=black] (0.62,1.09) circle (2pt);
			\draw [fill=black] (1,1) circle (2pt);
			\draw [fill=black] (1.48,1.05) circle (2pt);
			\draw [fill=black] (1.24,2.95) circle (2pt);
			\draw [fill=black] (2,3) circle (2pt);
			\draw [fill=black] (2.82,2.97) circle (2pt);
			\draw [fill=black] (2.58,1.05) circle (2pt);
			\draw [fill=black] (3,1) circle (2pt);
			\draw [fill=black] (3.62,1.01) circle (2pt);
			\draw [fill=black] (3.38,2.97) circle (2pt);
			\draw [fill=black] (4,3) circle (2pt);
			\draw [fill=black] (4.64,3.01) circle (2pt);
			\draw [fill=black] (4.66,1.07) circle (2pt);
			\draw [fill=black] (5,1) circle (2pt);
			\draw [fill=black] (5.52,1.01) circle (2pt);
			\draw [fill=black] (7,2) circle (2pt);
			\draw [fill=black] (8,2) circle (2pt);
			\draw [fill=black] (9,2) circle (2pt);
			\draw [fill=black] (10,2) circle (2pt);
			\draw [fill=black] (11,2) circle (2pt);
			\draw [fill=black] (12,2) circle (2pt);
			\draw [fill=black] (6.5,1) circle (2pt);
			\draw [fill=black] (7,1) circle (2pt);
			\draw [fill=black] (7.5,1) circle (2pt);
			\draw [fill=black] (7.5,3) circle (2pt);
			\draw [fill=black] (8,3) circle (2pt);
			\draw [fill=black] (8.5,3) circle (2pt);
			\draw [fill=black] (8.5,1) circle (2pt);
			\draw [fill=black] (9,1) circle (2pt);
			\draw [fill=black] (9.5,1) circle (2pt);
			\draw [fill=black] (9.48,2.97) circle (2pt);
			\draw [fill=black] (10,3) circle (2pt);
			\draw [fill=black] (10.54,2.99) circle (2pt);
			\draw [fill=black] (10.5,1) circle (2pt);
			\draw [fill=black] (11,1) circle (2pt);
			\draw [fill=black] (11.52,1.03) circle (2pt);
			\draw [fill=black] (11.5,3) circle (2pt);
			\draw [fill=black] (12,3) circle (2pt);
			\draw [fill=black] (12.5,3) circle (2pt);
		\end{scriptsize}
	\end{tikzpicture}
	\caption{Caterpillars tree with $(n,3)$ vertices}\label{fig2}
\end{figure} 

\begin{proposition}\label{projas}
 Albertson index of Caterpillar trees given by:
\[
\operatorname{irr}(C(n,m))= \begin{cases} m(m+1)n-2m+2, \quad &\text{if } n\geqslant 3, \\
m(m+1)n-2m, \quad &\text{if } n\in\{1,2\}.\end{cases}
\]
\end{proposition}
\begin{proof}

 We know $S_m$ is the graph star of order $m$, then Albertson index of star every edge combines the central vertex (degree $m$) with a pendant vertex (degree 1) given by $\operatorname{irr}=m(m-1) = m(m+1)n-2m$ for $n=1$. 

 For greater $n$ let sequence $x_1,x_2,\dots,x_n$ be a central path of caterpillar and $y_{i,1},\ldots,y_{i,m}$ pendant vertices adjacent to $x_i$ (see figure~\ref{gene}). Then $\deg x_1 = \deg x_n = m+1$ and $\deg x_i = m+2$ for $2\leqslant i\leqslant n-1$.

 \begin{figure}[H]
\centering
\begin{tikzpicture}[scale=1.5]
	\draw  (3,3)-- (6,3);
	\draw  (8,3)-- (11,3);
	\draw  (3,3)-- (2,2);
	\draw  (3,3)-- (2.4730903303069725,1.9875333744351633);
	\draw  (3,3)-- (3.6,1.99);
	\draw  (3,3)-- (4,2);
	\draw  (6,3)-- (5,2);
	\draw  (6,3)-- (5.401803618428293,2.048043979561637);
	\draw  (6,3)-- (6.24,2.05);
	\draw  (6,3)-- (6.62,2.05);
	\draw  (8,3)-- (7.543879039905456,2.0601461005869317);
	\draw  (8,3)-- (7.822227823487235,2.0117376164857528);
	\draw  (8,3)-- (8.66,2.01);
	\draw  (8,3)-- (9.310788709598485,2);
	\draw  (11,3)-- (10.44,2.05);
	\draw  (11,3)-- (10.8,2.05);
	\draw  (11,3)-- (11.5738853413286,2.0359418585363422);
	\draw  (11,3)-- (12,2);
	\draw (2.9329709292681723,3.5) node[anchor=north west] {$x_1$};
	\draw (6.00690966969303,3.5) node[anchor=north west] {$x_2$};
	\draw (7.967453275790774,3.5) node[anchor=north west] {$x_{n-1}$};
	\draw (10.811451716735032,3.5) node[anchor=north west] {$x_n$};
	\draw (10.31526475469795,2) node[anchor=north west] {$y_{n,1}$};
	\draw (10.750941111608558,2) node[anchor=north west] {$y_{n,2}$};
	\draw (10.968779290063864,2.2) node[anchor=north west] {$\dots$};
	\draw (12.045868061315094,2) node[anchor=north west] {$y_{n,m}$};
	\draw (7.386551466576628,2) node[anchor=north west] {$y_{n-1,1}$};
	\draw (8.015861759891953,2.2) node[anchor=north west] {$\dots$};
	\draw (9.129256894219067,2) node[anchor=north west] {$y_{n-1,m}$};
	\draw (1.8437800369916477,2) node[anchor=north west] {$y_{1,1}$};
	\draw (2.811949719015225,2.2) node[anchor=north west] {$\dots$};
	\draw (4,2) node[anchor=north west] {$y_{1,m}$};
	\draw (4.893514535365916,2) node[anchor=north west] {$y_{2,1}$};
	\draw (5.595437554833009,2.2) node[anchor=north west] {$\dots$};
	\draw (6.8,2) node[anchor=north west] {$y_{2,m}$};
	\draw (6.612015720957766,3.2) node[anchor=north west] {$\dots$};
	\begin{scriptsize}
		\draw [fill=black] (3,3) circle (2pt);
		\draw [fill=black] (6,3) circle (2pt);
		\draw [fill=black] (8,3) circle (2pt);
		\draw [fill=black] (11,3) circle (2pt);
		\draw [fill=black] (2,2) circle (2pt);
		\draw [fill=black] (2.4730903303069725,1.9875333744351633) circle (2pt);
		\draw [fill=black] (3.6,1.99) circle (2pt);
		\draw [fill=black] (4,2) circle (2pt);
		\draw [fill=black] (5,2) circle (2pt);
		\draw [fill=black] (5.401803618428293,2.048043979561637) circle (2pt);
		\draw [fill=black] (6.24,2.05) circle (2pt);
		\draw [fill=black] (6.62,2.05) circle (2pt);
		\draw [fill=black] (7.543879039905456,2.0601461005869317) circle (2pt);
		\draw [fill=black] (7.822227823487235,2.0117376164857528) circle (2pt);
		\draw [fill=black] (8.66,2.01) circle (2pt);
		\draw [fill=black] (9.310788709598485,2) circle (2pt);
		\draw [fill=black] (10.44,2.05) circle (2pt);
		\draw [fill=black] (10.8,2.05) circle (2pt);
		\draw [fill=black] (11.5738853413286,2.0359418585363422) circle (2pt);
		\draw [fill=black] (12,2) circle (2pt);
	\end{scriptsize}
\end{tikzpicture}
\caption{General case of Caterpillar tree} \label{gene}
\end{figure}
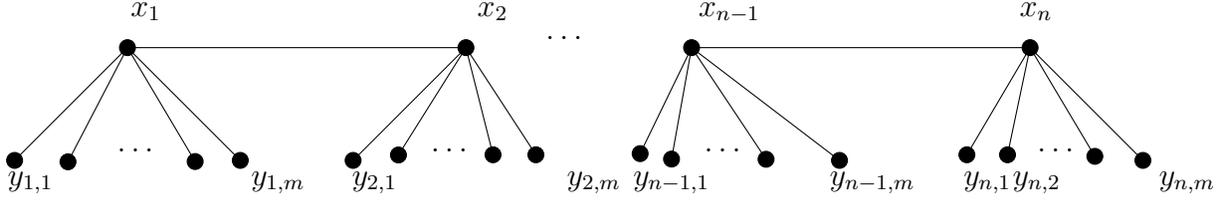

 Firstly, for $n=2$, we have depend main definition of Albertson index as: 
 \begin{gather*}
     \operatorname{irr}(G)=\sum_{j=1}^{m}\lvert \deg x_1-\deg y_{1,j}\rvert +\sum_{j=1}^{m}\lvert \deg x_2-\deg y_{2,j}\rvert=2m(m+1-1)=2m^2.
 \end{gather*}
Then we have $m(m+1)n-2m$ correct for $n=2$.

 Now let $n\geqslant 3$. So we can calculate $\operatorname{irr}$ as  we have $\sum_{j=1}^{m}\lvert\deg x_i-\deg y_{i,j}\rvert$, then we have: 
    \[
   \sum_{j=1}^{m}\lvert\deg x_i-\deg y_{i,j}\rvert=
    \begin{cases}
        m^2, & \quad \text{if } i=1, i=n,  \\
        m(m+1), & \quad \text{if } 2\leqslant i \leqslant n-1.\\
    \end{cases}
    \]
Therefore, as we have $\lvert\deg x_1-\deg x_2\rvert+\dots+\lvert\deg x_{n}-\deg x_{n-1}\rvert = 2$, the Albertson index is \[  \operatorname{irr}(C(n,m)) = 2m^2+(m+1)m(n-2)+2,\] as desire.

\end{proof}

\begin{proposition}
Let be $\operatorname{irr}(C(n,3))$ Albertson index of caterpillar tree of order $(n,3)$ vertices, then we have: 
\[
\operatorname{irr}_{\max}(C(n,m)) = \operatorname{irr}_{\min}(C(n,m)) = \operatorname{irr}(C(n,m)).
\]
This implies that we cannot specify $\operatorname{irr}_{\max}, \operatorname{irr}_{\min}$  in caterpillar tree $C(n,3)$, or in other terms, there are no such values but only the general value of the Albertson index $\operatorname{irr}$.
\end{proposition}
\begin{proposition}
Sigma index for caterpillar tree of order $(n,m)$ given by: 
\[
\sigma(C(n,m))=\begin{cases}
    2m^3, & \quad \text{if } n=2 \\
    2m^3+m-2,  & \quad \text{if } n\geqslant 2. \\ 
\end{cases}
\]
\end{proposition}
\begin{proof}
Assume the sequence $x_1,x_2,\dots,x_n$ and $y_{1,1},\dots, y_{i,j}$ pendant vertices adjacent to $x_i$. Then $\deg x_1 = \deg x_n = m+1$ and $\deg x_i = m+2$ for $2\leqslant i\leqslant n-1$. 
from Figure~\ref{fig2} we have in fact, for $n=2$, then Sigma index of caterpillar tree given as: $\sigma=2m^3$, for $n=3$, then we have: $\sigma=2m^3+m-2$, therefore, we can express from Figure~\ref{gene} that as: 
\[
\sigma(C(n,m))=\begin{cases}
    2m^3, & \quad \text{if } n=2 \\
    2m^3+m-2,  & \quad \text{if } n\geqslant 2. \\ 
\end{cases}
\]
Then we have $(\deg x_1-\deg x_2)^2+\dots+(\deg x_{n+1}-\deg x_n)^2$ that is hold to:
\[
\sigma=\left (\sum_{j=1}^n\deg x_i-\deg y_{i,j} \right )^2 \quad \forall i \in \mathbb{N}.
\]
\end{proof}
 \begin{proposition}{\label{Caterpillar}}
 For Caterpillar tree with path vertices with degrees $d_1,d_2,\dots , d_n$, we have:
 \[\operatorname{irr}(G)=\left( {{d_n} - 1} \right)^2 + \left( {d_1 - 1} \right)^2 + \sum\limits_{i = 2}^{n - 1} {\left( {{d_i} - 1} \right)\left( {{d_i} - 2} \right)} +\sum_{i=1}^{n-1}|d_i-d_{i+1}|.\]
 \end{proposition}
 \begin{hypothesize}{\label{hy.1}}
Let be a sequences with $n\geq 3$ and  let be order as: ${d_n} > d_1>\dots > d_2 > d_{n-1}$, the Caterpillar tree with such order has the  maximum value  of $\operatorname{irr}$ among all Caterpillar trees with such degrees sequence of path vertices.
\end{hypothesize}
\begin{proof}
 By using Proposition~\ref{Caterpillar} we have for $n=3$, $d_1\geq d_2 \geq d_3$, the maximum case is when $d_3$ is the degree of the central vertex of tree and from proposition~\ref{Caterpillar}, We can therefore generalise for example for $n=3$, this to any sequence of degree $n$ is correct given by ${d_n} > d_1>\dots > d_2 > d_{n-1}$ Such that:

 \[{\irr(G)=\left( {{d_n} - 1} \right)^2} + {\left( {d_1 - 1} \right)^2} + \sum\limits_{i = 2}^{n - 2} {\left( {{d_i} - 1} \right)\left( {{d_i} - 2} \right)} +\sum_{i=1}^{n-1}|d_i-d_{i+1}|.\]
 Now we need to prove that for $n+1$ with the sequence is $d_{n+1} \geq d_1\geq \dots \geq d_2 \geq d_n$, since $d_1\geq d_n$, then we can  simplify model the association as $(d_n-1)^2=(d_n-1)(d_n-2)+(d_n-1)$, further we have $d_1=d_n+r \mid  r\geq 0, r\in \mathbb{N}$, so that we have $(d_n-1)^2+(d_1-1)=(d_n)^2-d_n+r$ and also for $\sum_{i=2}^{n-2}(d_i-1)(d_i-2)$ we have $\sum_{i=2}^{n-2}((d_i)^2-3d_i+2)$, then we have: 
 \begin{align*}
     (d_i-1)(d_i-2)=& (d_i-1)^2+(d_i-1)\\
                   &=(d_i)^2-3d_i+2(n-3).
 \end{align*}
 where we consider $\sum_{i=2}^{n-2}2=2(n-3)$ Thus we can write it as: 
 \[
 \sum_{i=2}^{n-2}(d_i-1)(d_i-2)=\sum_{i=2}^{n-2}(d_i)^2-3\sum_{i=2}^{n-2}d_i+2(n-3).
 \]
 Then we have: 
 \[
 \irr(G)=(d_n-1)(d_n-2)+(d_n-1)+\sum_{i=2}^{n-2}(d_i)^2-3\sum_{i=2}^{n-2}d_i+2(n-3)+\sum_{i=1}^{n-1}\lvert d_i-d_{i+1}\rvert .
 \]
 Therefore, we suppose $\sum_{i=1}^{n-1} |d_i-d_{i+1}|-\sum_{i=2}^{n-1} d_i $ is max for $n$ when $d_n>d_1>\dots > d_2>d_{n-1}$, then  the question is why  we consider: 
  $$\sum_{i=1}^{n-1} |d_i-d_{i+1}|-\sum_{i=2}^{n-1} di$$ is max when $d_n>d_1>\dots > d_2>d_{n-1}$, so that, we need to prove the relationship is max for $n+1$. In this case we have: 
      \[
     \irr(G)=\sum_{i=1}^{n-1} |d_i-d_{i+1}| +(d_1+d_n)
      \] 
Assume $ \sum_{i=1}^{n-1} |d_i-d_{i+1}|+d_1+d_n$ is max when $d_n>d_1>\dots > d_2>d_{n-1}$, Therefore $ \sum_{i=1}^{n-1} |d_i-d_{i+1}|+d_1+d_{n+1} \quad \text{is max when} \quad d_n>d_1>\dots > d_2>d_{n-1}$, we noticed that $d_1+d_{n+1}$ is maximum value when $d_{n+1}\geqslant d_1\geqslant \dots \geqslant d_2 \geqslant d_n$, then we need to prove $\sum_{i=1}^{n-1} |d_i-d_{i+1}|$ is maximum value when $d_{n+1}\geqslant d_1\geqslant \dots \geqslant d_2 \geqslant d_n$, therefore, we have 
\begin{align*}
    \irr(G)&=\sum_{i=1}^{n-1} |d_i-d_{i+1}|+d_1+d_{n+1}\\
    &=|d_1-d_2|+|d_2-d_3|+\dots +|d_{n-1}-d_{n}|+d_1+d_{n+1}
\end{align*}
Now, we have two cases that we can discuss as follows, In order to compare both $d_i$ and $d_{i+1}$ and to determine the maximised value of the term $\sum_{i=1}^{n-1} |d_i-d_{i+1}|$  according to $\operatorname{irr}(G)$ as: 
\case{1} If $d_i\geq d_{i+1}$, then we have $\operatorname{irr}(G)=2d_1 - d_n + d_{n+1}$ is maximum, notice $2d_1$ comes from combining the two $d_1$ terms where $d_n + d_{n+1}$ is maximum when $d_{n+1}\geqslant d_1\geqslant \dots \geqslant d_2 \geqslant d_n$.
\case{2} If $d_i\leq d_{i+1}$, then we have $\irr (G)=d_n + d_{n+1}$ is the maximum, so that in this case the term $\sum_{i=1}^{n-1} |d_i-d_{i+1}|$ is combined with the maximum value $d_1+d_{n+1}$.

Thus, we have $$\irr(G)=
\begin{cases}
    2d_1 - d_n + d_{n+1} & \quad \text{ if } d_i\geq d_{i+1},\\
    d_n + d_{n+1} & \quad \text{ if } d_i\leq d_{i+1}.
\end{cases}
$$ is the maximum when $d_{n+1}\geqslant d_1\geqslant \dots \geqslant d_2 \geqslant d_n$ and $d_1 + d_{n+1}$ is maximum.

As desire.
\end{proof}
\subsection{Sigma Index and Albertson Index Among Trees}

\begin{hypothesize}~\label{hy.sigma}
Let $T$  be a tree of order $n\geq 2$, a degree sequence $d=(d_1,d_2,d_3)$ where $d_3\geq d_2\geq d_1$, then Sigma index given by: 
\[
\sigma(T)=(d_1-1)^3+\sum_{i=1}^{3}(d_i-1)(d_i-2)+(d_3-1)^3.
\]
\end{hypothesize}
\begin{proof}
For the sequence $d(d_1,d_2,d_3)$, we will investigate all potential instances of this sequence that describe the positioning of each of these vertices as shown in Figure~\ref{fig.2.4.22}, which depicts the three principal cases we are considering as:
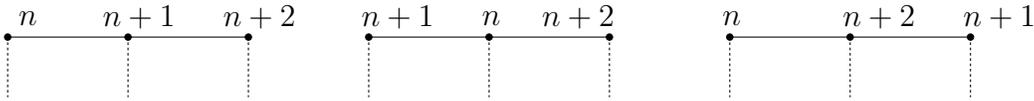
\begin{figure}[H]
    \centering
\begin{tikzpicture}[scale=.8]
\draw   (1,2)-- (3,2);
\draw   (3,2)-- (5,2);
\draw   (7,2)-- (9,2);
\draw   (9,2)-- (11,2);
\draw   (13,2)-- (15,2);
\draw   (15,2)-- (17,2);
\draw [ dash pattern=on 1pt off 1pt] (1,2)-- (1,1);
\draw [ dash pattern=on 1pt off 1pt] (3,2)-- (3,1);
\draw [ dash pattern=on 1pt off 1pt] (5,2)-- (5,1);
\draw [ dash pattern=on 1pt off 1pt] (7,2)-- (7,1);
\draw [ dash pattern=on 1pt off 1pt] (9,2)-- (9,1);
\draw [ dash pattern=on 1pt off 1pt] (11,2)-- (11,1);
\draw [ dash pattern=on 1pt off 1pt] (13,2)-- (13,1);
\draw [ dash pattern=on 1pt off 1pt] (15,2)-- (15,1);
\draw [ dash pattern=on 1pt off 1pt] (17,2)-- (17,1);

\draw (1,2.6) node[anchor=north west] {$n$};
\draw (2.4,2.7) node[anchor=north west] {$n+1$};
\draw (4.4,2.7) node[anchor=north west] {$n+2$};
\draw (6.7,2.7) node[anchor=north west] {$n+1$};
\draw (8.7,2.6) node[anchor=north west] {$n$};
\draw (9.7,2.7) node[anchor=north west] {$n+2$};
\draw (12.7,2.6) node[anchor=north west] {$n$};
\draw (14.7,2.7) node[anchor=north west] {$n+2$};
\draw (16.7,2.7) node[anchor=north west] {$n+1$};
\begin{scriptsize}
\draw [fill=black] (1,2) circle (1.5pt);
\draw [fill=black] (3,2) circle (1.5pt);
\draw [fill=black] (5,2) circle (1.5pt);
\draw [fill=black] (7,2) circle (1.5pt);
\draw [fill=black] (9,2) circle (1.5pt);
\draw [fill=black] (11,2) circle (1.5pt);
\draw [fill=black] (13,2) circle (1.5pt);
\draw [fill=black] (15,2) circle (1.5pt);
\draw [fill=black] (17,2) circle (1.5pt);
\end{scriptsize}
\end{tikzpicture}
\caption{describe sequence $d(d_1,d_2,d_3)$.}
    \label{fig.2.4.22}
\end{figure}
\case{1}
in this case we have first recognize as $d(d_1,d_2,d_3)$, then we have: 
\begin{align*}
\sigma(T)=& d_1(d_1-1)^2+d_2(d_2-d_3)^2+d_3(d_3-d_2)^2+2\\
&=d_1^3 - 2d_1^2 + d_1 + d_2^3 + d_3^3 - 4d_2d_3^2 + 2.
\end{align*}
\case{2}
    in this case we have the sequence recognized as $d=(d_2,d_1,d_3)$, notice that for $n\geqslant 2$, then  we have
    \begin{align*}
        \sigma(T)=&2(d_2)^2+d_1(d_1-1)^2+d_2(d_3-d_2)^2+2\\
        &=d_1^3 - 2d_1^2 + d_1 + d_2^3 + 2d_2^2(1 - d_3) + d_2d_3^2+2.
    \end{align*}
\case{3}
  in this case we have the sequence recognized as $d=(d_1,d_3,d_2)$, then we notice This case is similar to Case 2, so that, for $n\geqslant 2$, we have: 
  \begin{align*}
      \sigma(T)=&d_1(d_1-1)^2+d_2(d_3-1)^2+d_1(d_2-d_1)^2+5\\
      &=2d_1^3 - 2d_1^2(1 + d_2) + d_1 + d_2d_3^2 - 2d_2d_3 + d_2 + 5.
  \end{align*}
Therefore, from Case 1 and Case 2 we have: 
\[
\sigma(T)=d_3^3 - 2d_2^2 + 2d_2^2d_3 - 5d_2d_3^2.
\]
and from Case 1 and Case 2, we have: 
\[
\sigma(T)=-d_1^3 + 2d_1^2d_2 + d_2^3 + d_3^3 - 5d_2d_3^2 + 2d_2d_3 - d_2 - 3.
\]
from Case 2 and Case 3, we have: 
\[
\sigma(T)=-d_1^3 + 2d_1^2d_2 + d_2^3 + 2d_2^2 - 2d_2^2d_3 + 2d_2d_3 - d_2 - 3.
\]
Thus, we can clearly see that the first case is the maximum case and the second case is the minimum case for the sigma index, and we denote this by
\[
\sigma(T)=\begin{cases}
   \sigma_{\max}(T)= d_1^3 - 2d_1^2 + d_1 + d_2^3 + d_3^3 - 4d_2d_3^2 & \quad \text{ if } d_3\geqslant 1,\\
    \sigma_{\min}(T)= d_1^3 - 2d_1^2 + d_1 + d_2^3 + 2d_2^2(1 - d_3) + d_2d_3^2 & \quad \text{ if } d_1\geqslant 1.
\end{cases}
\]
As desire.
\end{proof}
 \begin{hypothesize}
Let $\tilde{d}$ be a degree sequence $\tilde{d}=(d_1,d_2,d_3,d_4)$ where $d_1 \ge d_2 \ge d_3 \ge d_4$. defined numbers as $(a,b,c,d)$ as $d > a \geq b \geq c$ we have: 
  \[
  \operatorname{irr}(G)=\sum_{i=1}^{n-1}(d_i-1)^2+ \left( {d - a} \right) + \left( {d - b} \right) + \left( {d - c} \right) + \left( {d - 1} \right)\left( {d - 3} \right).
 \]
 \end{hypothesize}
  \begin{hypothesize}~\label{hy.sigma2}
Let $T$ be a tree of order $n>0$ and a sequence $d=(d_1,d_2,d_3,d_4)$ where $d_4\geqslant d_3\geqslant d_2\geqslant d_1$, then Sigma index is: 
\[
\sigma(T)=\sum_{i=1}^{4} d_i^3 + 2\sum_{i=1}^{4} d_i^2 + \sum_{i=1}^{4} d_i - 2\sum_{i=1}^{3} d_i d_{i+1}.
\]
if and only if the inequality holds the terms $d_1>0, d_2=d_1+1, d_3=d_2+1, d_4=d_3+1$.
\end{hypothesize}
\section{Analyse The Results For Topological Indices}
In fact, both indices (Sigma Index and Albertson Index) have common preliminary factors based on the definition of each of them. in this section we will examine each index on two specific types of tree classes, the first type is normal trees and the second type is caterpillar trees, as we note that in Table~\ref{tab1und} we show the Albertson Index on trees and caterpillar trees and the difference between both effects is clear as shown.
\begin{corollary}
For integer number $n\geqslant 3$, and a caterpillar tree $C(n,n)$, then we have: 
\[
\operatorname{irr}(C)=n^3+n^2-2n+2.
\]
\end{corollary}
\begin{proof}
Actually, for $n\in[3,6]$ we have $\operatorname{irr}(C)=13n^2 - 49n + 62$, but for $n\in[3, +\infty[$, this relationship will be corrected for some value not for all value, so that we have in Proposition~\ref{projas} for $n\geqslant 3$, then $\irr(C)=m(m+1)n-2m+2$, but in our case $n=m$, then this relation become: 
\begin{align*}
    \irr(C)&=m(m+1)n-2m+2\\
    &=n^2(n+1)-2n+2\\
    &=n^3+n^2-2n+2.
\end{align*}
\end{proof}

\begin{corollary}
For caterpillar tree $C(n,m)$, with $n\in[3,10], m\in [3,10]$, the we have: 
\[
\begin{cases}
    \max(\operatorname{irr}(C), \sigma(C))=\sigma(C) & \quad \text{ if } m\geqslant n, \\
     \max(\operatorname{irr}(C), \sigma(C))=\operatorname{irr}(C) & \quad \text{ if } n\geqslant 5, m=3,\\
     \max(\operatorname{irr}(C), \sigma(C))=\operatorname{irr}(C) & \quad \text{ if } n \in \{7,9\}, m\in \{4,5\}.
\end{cases}
\]
\end{corollary}
\begin{table}[H]
    \centering
\begin{tabular}{|c|c|c|c|c|c|c|c|c|c|c|c|}
\hline
$n$ & $m$ & $\operatorname{irr}$ & $\sigma$ & $\sigma-\operatorname{irr}$ & $\max(\sigma,\operatorname{irr})$ & $n$ & $m$ & $\operatorname{irr}$ & $\sigma$ & $\sigma-\operatorname{irr}$ & $\max(\sigma,\operatorname{irr})$\\ \hline \hline
3 & 3 & 32 & 55 & 23 & 55 & 5 & 3 & 56 & 55 & -1 & 56 \\ \hline
3 & 6 & 116 & 436 & 320 & 436 & 5 & 6 & 200 & 436 & 236 & 436  \\ \hline
3 & 7 & 156 & 691 & 535 & 691  & 5 & 7 & 268 & 691 & 423 & 691\\ \hline
3 & 9 & 254 & 1465 & 1211 & 1465 & 5 & 9 & 434 & 1465 & 1031 & 1465\\ \hline
4 & 3 & 44 & 55 & 11 & 55 & 6 & 3 & 68 & 55 & -13 & 68\\ \hline
4 & 4 & 74 & 130 & 56 & 130 & 6 & 4 & 114 & 130 & 16 & 130 \\ \hline
4 & 7 & 212 & 691 & 479 & 691 & 6 & 7 & 324 & 691 & 367 & 691\\ \hline
4 & 9 & 344 & 1465 & 1121 & 1465 & 6 & 9 & 524 & 1465 & 941 & 1465\\ \hline
4 & 10 & 422 & 2008 & 1586 & 2008 & 6 & 10 & 642 & 2008 & 1366 & 2008\\ \hline
7 & 3 & 80 & 55 & -25 & 80 & 9 & 3 & 104 & 55 & -49 & 104\\ \hline
7 & 5 & 202 & 253 & 51 & 253 & 9 & 5 & 262 & 253 & -9 & 262\\ \hline
7 & 7 & 380 & 691 & 311 & 691 & 9 & 7 & 492 & 691 & 199 & 691\\ \hline
7 & 8 & 490 & 1030 & 540 & 1030 & 9 & 8 & 634 & 1030 & 396 & 1030 \\ \hline
7 & 9 & 614 & 1465 & 851 & 1465& 9 & 9 & 794 & 1465 & 671 & 1465  \\ \hline
7 & 10 & 752 & 2008 & 1256 & 2008 & 9 & 10 & 972 & 2008 & 1036 & 2008\\ \hline
8 & 3 & 92 & 55 & -37 & 92& 10 & 3 & 116 & 55 & -61 & 116 \\ \hline
8 & 5 & 232 & 253 & 21 & 253 & 10 & 5 & 292 & 253 & -39 & 292\\ \hline
8 & 7 & 436 & 691 & 255 & 691& 10 & 7 & 548 & 691 & 143 & 691 \\ \hline
8 & 8 & 562 & 1030 & 468 & 1030& 10 & 8 & 706 & 1030 & 324 & 1030 \\ \hline
8 & 10 & 862 & 2008 & 1146 & 2008&  10 & 10 & 1082 & 2008 & 926 & 2008\\ \hline
\end{tabular}
\caption{Compared Value of Sigma Index and Albertson Index on Caterpillar Tree.}
    \label{tab1und}
\end{table}
 \begin{example}
 	let a sequence $d_1, d_2, d_3, d_4 ; d_1 \geq d_2 \geq d_3 \geq d_4$ and we will suppose sequences is: $(10,8,3,2)$, We will take 2 cases:
 	\begin{enumerate}
 		\item[Case 1] $(3,10,8,2)$, in this case we have: $\irr=134$.
 		\item[Case 2] $(10,3,2,8)$, in this case we have: $\irr=146$.
 	\end{enumerate}
 		we suppose: $a=d_1,b=d_2, c=d_3, d= d_4$. 	So that we obtain on:
 	\begin{enumerate}
 		\item if $a=10, b=8,c=3,d=2$, in this case we have: $\irr=134, \sigma=1036$.
 		\item if $a=8,b=10,c=3,d=2$, in this case we have: $\irr=134,\sigma=1048$.
 		\item if $a=3,b=10,c=8,d=2$, in this case we have: $\irr=144,\sigma=1148$.
 		\item if $a=2,b=10,c=8,d=3$, in this case we have: $\irr=146,\sigma=1156$.
 	\end{enumerate}
 	
 	\begin{table}[H]
 		\centering
 		\begin{tabular}{|c|c|c|c|}
 			\hline
 			\multicolumn{2}{|c|}{$\irr$} &  \multicolumn{2}{|c|}{$\sigma$} \\
 			\hline
 			$\max$ & $\min$ & $\max$ & $\min$ \\
 			\hline
 			146   & 134    & 1156 &   1036\\
 			\hline
 		\end{tabular}
 	\end{table}	
 \end{example}

\begin{example}
	let a sequence $d_1, d_2, d_3, d_4 ; d_1 \geq d_2 \geq d_3 \geq d_4$ and we will suppose sequences is: $(8,5,3,2)$, We will take 4 cases as:
	\begin{enumerate}
		\item[Case 1] $(8,5,3,2)$, in this case we have: $\irr=70$.
		\item[Case 2] $(2,8,5,3)$, in this case we have: $\irr=70$.
		\item[Case 3] $(8,3,2,5)$, in this case we have: $\irr=76$.
		\item[Case 4] $(8,2,5,3)$, in this case we have: $\irr=76$.
	\end{enumerate}
	\[ 
	\irr =\left\lbrace  \begin{array}{c}
		\irr_{\max}=76 \\
		\irr_{\min}= 70.
	\end{array}\right.\]
\end{example}
In order to reduce the number of examples to an infinite number of examples, we can use Python, where we create a code that helps us enter a sequence of $4$ elements and gives us for $\irr$ with $n=4$ as we show below.
\begin{verbatim}
 import math
 def irr_sigma_max(a,b,c,d):
 	x=(a-1)**2+(b-1)**2 +(c-1)**2+(d-1)**2
 	y=(a + b - c - (3 * d )+ 2)
 	return (x+y)
def irr_sigma_min(a,b,c,d):
	z=(a-1)**2+(b-1)**2 +(c-1)**2+(d-1)**2
	m=(a-b-c-d+2)
	return (z+m)
# Function to check if the list is in descending order
def is_sorted_descending(lst):
     for i in range(len(lst) - 1):
 		if lst[i] < lst[i + 1]:
 		return False
 	return True
a=int(input("Enter an element: "))
b=int(input("Enter an element: "))
c=int(input("Enter an element: "))
d=int(input("Enter an element: "))
list=[a,b,c,d]
#list.sort()
print ("d_4=", list)
# Check if the list is sorted
if is_sorted_descending(list):
    print("irr_sigma_max=",irr_sigma_max(a,b,c,d))
    print ("irr_sigma_min=", irr_sigma_min(a,b,c,d))
 else:
    print("The list is NOT sorted from largest to smallest.")	
        \end{verbatim}
\section{Conclusion}
Topological indices are very important in graphs, especially on trees, and through this paper we presented both the sigma index and the Albertson index. We studied these indices on trees, including caterpillar trees, and analysed these results by creating a table to demonstrate these results.

\section*{Acknowledgments}

I would like to extend my deepest gratitude to Prof. Alexei Kanel-Belov. Also, I grateful for Dr. Duaa Abdullah for positive feedback.

\end{document}